\documentclass[a4paper, 11pt]{article}
\usepackage{amsmath, amssymb,amscd, amsthm}
\usepackage[mathscr]{eucal}
\usepackage{graphics}
\usepackage{fullpage}
\usepackage{url}
\newcommand\cyr{%
 \renewcommand\rmdefault{wncyr}%
 \renewcommand\sfdefault{wncyss}%
 \renewcommand\encodingdefault{OT2}%
\normalfont\selectfont} \DeclareTextFontCommand{\textcyr}{\cyr}

\newtheorem{theorem}{Theorem}
\newtheorem{lemma}[theorem]{Lemma}
\newtheorem{corollary}[theorem]{Corollary}

\newtheorem{conjecture}[theorem]{Conjecture}

\def\Z{\mathbb Z}

\def\mod{\operatorname{mod}}

\newtheorem*{cor8}{Corollary 8}
\newtheorem*{thm4}{Theorem 4}
\newtheorem*{thm5}{Theorem 5}

\long\def\symbolfootnote[#1]#2{\begingroup%
\def\thefootnote{\fnsymbol{footnote}}\footnote[#1]{#2}\endgroup}

\begin{document}

\title{\textbf{A Variational Barban-Davenport-Halberstam Theorem}}

\author{Allison Lewko \thanks{Supported by a Microsoft Research PhD Fellowship.}
\and Mark Lewko}

\date{}

\maketitle

\begin{abstract} We prove variational forms of the Barban-Davenport-Halberstam Theorem and the large sieve inequality. We apply our result to prove an estimate for the sum of the squares of prime differences, averaged over arithmetic progressions.
\end{abstract}

\section{Introduction}
The prime number theorem implies the asymptotic $\psi(x) \sim x$, while the Riemann hypothesis predicts a bound of $|\psi(x) - x| \ll_\epsilon x^{\frac{1}{2} + \epsilon}$ on the error term. This extends naturally to arithmetic progressions, where the asymptotic $\psi(x;q,a) \sim \frac{x}{\phi(q)}$ holds for all coprime $a$ and $q$.
We recall that
\[\psi(x;q,a) := \sum_{\substack{n \leq x \\ n \equiv a \mod q}} \Lambda(n).\]
Under the Generalized Riemann hypothesis, one obtains the error bound $\left|\psi(x;q,a)-\frac{x}{\phi(q)}\right|\ll_\epsilon x^{\frac{1}{2} + \epsilon}$. The stronger bound of $\ll_\epsilon x^{\frac{1}{2} +\epsilon}\phi(q)^{-\frac{1}{2}}$ is also conjectured. (For further definitions, see Section \ref{sec:prelim}. For background material, see \cite{D}.)

An unconditional bound on the averaged error term for this is provided by the Barban-Davenport-Halberstam Theorem \cite{D}, which states:
\begin{theorem}\label{thm:BHD} (Barban-Davenport-Halberstam)
Let $A > 0$. For all positive real numbers $x$ and $Q$ satisfying $x (\log(x))^{-A} \leq Q \leq x$,
\[\sum_{q \leq Q} \sum_{\substack{a\leq q \\ (a,q)=1}} \left(\psi(x;q,a)-\frac{x}{\phi(q)}\right)^2 \ll_A xQ\log(x).\]
\end{theorem}
We note that this holds also for the quantity $\theta(x; q,a)$, since the differences between $\psi(x;q,a)$ and $\theta(x;q,a)$ are of lower order.

An even stronger bound is due to Montgomery \cite{M}, refining work of Uchiyama \cite{U}, (see also the refinement of Hooley \cite{H}):
\begin{theorem}\label{thm:MH}
Let $A > 0$. For all positive real numbers $x$ and $Q$ satisfying $x (\log(x))^{-A} \leq Q \leq x$,
\[\sum_{q\leq Q} \max_{y \leq x} \sum_{\substack{a\leq q\\(a,q)=1}} \left(\psi(y;q,a)- \frac{y}{\phi(q)}\right)^2 \ll_A xQ\log(x).\]
\end{theorem}
We note that the quantity on the left has potentially increased compared to the quantity in Theorem \ref{thm:BHD}, while the bound on the right is the same, up to the implicit constant.

Another variant of Theorem \ref{thm:BHD} is due to Uchiyama \cite{U}:
\begin{theorem}\label{thm:U}
Let $A > 0$. For all positive real numbers $x$ and $Q$ satisfying $x (\log(x))^{-A} \leq Q \leq x$,
\[\sum_{q \leq Q} \sum_{\substack{a\leq q \\ (a,q)=1}} \max_{y\leq x} \left(\psi(y;q,a)-\frac{y}{\phi(q)}\right)^2 \ll_A x Q \log^3(x).\]
\end{theorem}
This is incomparable to Theorems \ref{thm:BHD} and \ref{thm:MH}, since the quantity being bounded is larger and the bound obtained is worse. Hooley, in \cite{H2}, has announced a refinement to the $\log^3(N)$ for certain values of $Q$. This seems to have not yet appeared, however.

We work with the function $\theta$ instead of $\psi$ because it is more convenient for our purposes, though this is a minor difference.
To further refine our understanding of the deviation of $\theta(x;q,a)$ from its average value of $\frac{x}{\phi(q)}$, we introduce a variational operator in place of the maximal one in Theorem \ref{thm:U}. Letting $\{c_n\}_{n=1}^N$ be a finite sequence of complex numbers and letting $\mathcal{P}_N$ denote the set of partitions of $[N]:= \{1,2, \ldots, N\}$ into disjoint intervals, we define the $r$-variation of the sequence to be:
\[\left|\left|\{c_n\}_{n=1}^N\right|\right|_{V^r} := \max_{\pi \in \mathcal{P}_N} \left( \sum_{I \in \pi} \left|\sum_{n \in I} c_n\right|^r \right)^{\frac{1}{r}}.\]

We can think of $\theta(x;q,a)$ as a sum over a sequence $\{b_n\}_{n=1}^N$, where $N = \lfloor x \rfloor$ and
\[b_n := \left\{
    \begin{array}{ll}
      \log(n), & \hbox{$n$ prime;} \\
      0, & \hbox{otherwise.}
    \end{array}
  \right.\]
For an interval $I$, we define
\[\theta(I;q,a):= \sum_{n \in I} b_n.\]
Letting $|I|$ denote the number of integers contained in $I$, we then have that
\[ \max_{\pi \in \mathcal{P}_N} \sum_{I \in \pi} \left( \theta(I; q, a) - \frac{|I|}{\phi(q)}\right)^2\]
is the square of the 2-variation of the sequence $\{b_n-1/\phi(q)\}_{n=1}^N$.

Our main result is an upper bound on this quantity, summing over $q\leq Q$ and $a$ coprime to $q$ as in the above theorems.
This is a strengthening of Theorem \ref{thm:U}, since we obtain the same bound (up to the implied constant) on a larger quantity. To simplify our notation, we let $\mathcal{P}_x$ denote the set of partitions of $\{1,\ldots, \lfloor x\rfloor \}$ into disjoint intervals. We prove:
\begin{theorem}\label{thm:VarBHD} Let $A > 0$. For all positive real numbers $x$ and $Q$ satisfying $x (\log(x))^{-A} \leq Q \leq x$,
\[\sum_{q \leq Q} \sum_{\substack{ a\leq q \\ (a,q) = 1}} \max_{\pi \in \mathcal{P}_x} \sum_{I \in \pi} \left( \theta(I; q, a) - \frac{|I|}{\phi(q)}\right)^2 \ll_A xQ \log^3(x).\]
\end{theorem}

We also establish a variant of this, obtaining a better bound by allowing the partition to depend only on $q$ and not on $a$:
\begin{theorem}\label{thm:VarBHD2}
Let $A > 0$. For all positive real numbers $x$ and $Q$ satisfying $x (\log(x))^{-A} \leq Q \leq x$,
\[\sum_{q \leq Q} \max_{ \pi \in \mathcal{P}_x} \sum_{\substack{a\leq q \\ (a,q)=1}} \sum_{I \in \pi} \left(\theta(I;a,q)-\frac{|I|}{\phi(q)}\right)^2 \ll_A xQ \log^2(x).\]
\end{theorem}
In comparison to Theorem \ref{thm:MH}, this maximizes over partitions instead of restricting to partial sums, but the bound obtained is worse by a multiplicative $\log(x)$ factor.

The introduction of this maximum over partitions allows us to apply our theorem to prove a weakened, averaged version of a conjecture made by Erd\H{o}s. We let $p_i$ denote the $i^{th}$ prime. Erd\H{o}s made the following conjecture:
\begin{conjecture}(Erd\H{o}s, \cite{E})
\[
\sum_{p_{i+1} \leq x} (p_{i+1}-p_i)^2 \ll x \log(x).
\]
\end{conjecture}
This asymptotic is heuristically suggested by the prime number theorem, which implies the reverse inequality. Assuming the Riemann Hypothesis, Selberg obtained the bound
\[ \sum_{p_{i+1} \leq x} (p_{i+1}-p_i)^2 \ll x \log^3(x).\]

It is natural to extend the conjecture to arithmetic progressions. Fixing $a,q$ such that $(a,q) =1$, we let $p_i^{a,q}$ denote the $i^{th}$ prime congruent to $a$ modulo $q$. One then formulates the conjecture as:
\begin{conjecture} For any $a,q$ such that $(a,q) = 1$,
\[ \sum_{p_i^{a,q} \leq x} \left(\frac{p_{i+1}^{a,q} - p_i^{a,q}}{\phi(q)}\right)^2 \ll \frac{x \log(x)}{\phi(q)}.\]
\end{conjecture}

If we then sum over all $q \leq Q$ and all $a$ coprime to $q$, we would expect to get $\ll Qx \log(x)$. We derive the following weaker bound:

\begin{corollary}\label{cor:erdos}Let $A > 0$. For all positive real numbers $x$ and $Q$ satisfying $x (\log(x))^{-A} \leq Q \leq x$,
\[\sum_{q \leq Q} \sum_{\substack{a \leq q \\ (a,q)=1}} \sum_{p_{i+1}^{a,q} \leq x} \left( \frac{p_{i+1}^{a,q}-p_i^{a,q}}{\phi(q)}\right)^2 \ll Qx \log^3(x).\]
\end{corollary}
This can be viewed as an averaged, unconditional version of Selberg's bound, and is easily obtained from Theorem \ref{thm:VarBHD}.

More generally, the study of variational quantities introduces new and interesting questions. For example, is there an elementary function $f$ such that
\begin{equation}\label{eq:asym}
\max_{\pi \in \mathcal{P}_x} \sum_{I \in \pi} \left(\sum_{n \in I} (\Lambda(n)-1)\right)^2 \sim f(x)?
\end{equation}

The prime number theorem gives an (asymptotic) lower bound of $x\log(x)$ on this quantity. We note, however, that one cannot hope to have $f(x) = x\log(x)$. This follows from the work of Cheer and Goldston \cite{CG}, who proved
\begin{theorem}\label{thm:cg}
For any $\epsilon > 0$, there exists an $X_0$ such that for all $x > X_0$,
\[\sum_{p_{i+1} \leq x} |p_{i+1}-p_{i}|^{2} \geq (193/192 - \epsilon) x \log x.\]
\end{theorem}
(Note, as seen from Lemma \ref{lem:basicBound} below, the contribution to (\ref{eq:asym}) from prime powers is of lower order.) This does not rule out the possibility of $f(x) = Cx\log(x)$ for some larger $C$, for example.

\section{Preliminaries}\label{sec:prelim}
We first recall some standard definitions. When $q$ is a positive integer, $\phi(q)$ denotes the Euler totient function. For positive integers $a$ and $q$, $(a,q)$ denotes the g.c.d. of $a$ and $q$.

For a positive real number $x$, we define
\[\psi(x) = \sum_{n \leq x} \Lambda(n) = \sum_{p^{\alpha} \leq x} \log(p),\]
\[\theta(x) = \sum_{p \leq x} \log (p).\]
Here, $\log$ denotes the natural logarithm. The latter sum for $\psi$ is over prime powers $p^{\alpha}$, while the sum for $\theta$ is over primes $p$. $\Lambda(n)$ denotes the von Mangoldt function, which is equal to $\log(p)$ whenever $n$ is a power of a prime $p$ and equal to 0 otherwise.

Letting $a$ and $q$ be positive integers, we similarly define
\[\psi(x;q,a) = \sum_{\substack{n \leq x \\ n \equiv a \mod q}} \Lambda(n),\; \; \; \theta(x; q,a) = \sum_{\substack{p\leq x \\ p \equiv a \mod q}} \log (p).\]
Letting $I$ be an interval, we also define
\[\psi(I;q,a) = \sum_{\substack{n \in I \\ n \equiv a \mod q}} \Lambda(n), \; \; \; \theta(I; q,a) = \sum_{\substack{p \in I \\ p \equiv a \mod q}} \log (p).\]
The size of an interval $I$ is defined to be the number of integers it contains, and is denoted by $|I|$.
For a fixed positive real number $x$, we let $\mathcal{P}_x$ denote the set of all partitions of $[1, \lfloor x \rfloor]$ into intervals. Thus an element $\pi \in \mathcal{P}_x$ is a collection of disjoint intervals whose union is the interval from 1 to $\lfloor x \rfloor$.

We recall the prime number theorem, which states that $\psi(x) \sim x$. We will later also use the following standard fact (we include the short proof here for completeness):
\begin{lemma}\label{lem:basicBound} For $x \geq 2$, $\theta(x) = \psi(x) + O(x^{\frac{1}{2}})$.
\end{lemma}

\begin{proof} Since $p^{\alpha} \leq x$ holds if and only if $p  \leq x^{\frac{1}{\alpha}}$, we have
\[\psi(x) = \sum_{\alpha =1}^{\infty} \theta \left(x^{\frac{1}{\alpha}}\right) \text{   and   }
\psi(x) - \theta(x) = \sum_{\alpha \geq 2} \theta \left(x^{\frac{1}{\alpha}}\right).\]
Noting that $x^{\frac{1}{\alpha}} \geq 2$ only for $\alpha = O(\log x)$ and $\theta(x^{\frac{1}{\alpha}}) \leq \psi(x^{\frac{1}{\alpha}}) \ll x^{\frac{1}{\alpha}}$, we see this is $\ll x^{\frac{1}{2}} + x^{\frac{1}{3}} \log x \ll x^{\frac{1}{2}}$.
\end{proof}

\section{A Variational Form of the Barban-Davenport-Halberstam \\ Theorem}
We now prove:

\begin{thm4} Let $A > 0$. For all positive real numbers $x$ and $Q$ satisfying $x (\log(x))^{-A} \leq Q \leq x$,
\[\sum_{q \leq Q} \sum_{\substack{ a\leq q \\ (a,q) = 1}} \max_{\pi \in \mathcal{P}_x} \sum_{I \in \pi} \left( \theta(I; q, a) - \frac{|I|}{\phi(q)}\right)^2 \ll_A xQ \log^3(x).\]
\end{thm4}

We will deduce this by combining the proof of the standard Barban-Davenport-Halberstam theorem with some combinatorial arguments and a variational form of the Siegel-Walfisz Theorem that is developed in the following subsection.

For a fixed positive integer $q$, we consider Dirichlet characters modulo $q$. A function $\chi: \Z_q^* \rightarrow \mathbb{C}$ is called a Dirichlet character modulo $q$ if it is a group homomorphism. We can extend such a $\chi$ to be a function from $\Z$ to $\mathbb{C}$ by defining $\chi(n)$ to be equal to the value of the character on the residue class of $n$ modulo $q$ when $n$ is coprime to $q$ and 0 otherwise. From now on, we will consider Dirichlet characters to be functions on $\Z$. A character $\chi$ mod $q$ is said to be \emph{primitive} if its period as a function on $\Z$ is precisely $q$ (conversely it is non-primitive if it has a smaller period dividing $q$).

We fix positive integers $M$ and $N$. Given a Dirchlet character $\chi$ mod $q$ and complex numbers $\{a_n\}_{n=M+1}^{M+N}$, we define
\[T(\chi) = \sum_{n=M+1}^{M+N} a_n \chi(n).\] More generally, for any interval $I\subseteq [M+1,M+N]$, we define
\[T(\chi,I) = \sum_{n \in I} a_n \chi(n).\]

The large sieve inequality \cite{D} states:
\begin{theorem}\label{thm:standardLargeSieve} (The Large Sieve Inequality)
For any positive integers $Q, M, N$ and complex numbers $\{a_n\}_{n=M+1}^{M+N}$:
\[
\sum_{q \leq Q} \frac{q}{\phi(q)} \sideset{}{^*}\sum_{\chi \mod q} |T(\chi)|^2 \ll (N + Q^2) \sum_{n=M+1}^{M+N} |a_n|^2.
\]
Here, the the inner sum $\sum_{\chi}^*$ is over the \textbf{primitive} characters modulo $q$ (this is what the $*$ superscript signifies).
\end{theorem}

In our proof of Theorem \ref{thm:VarBHD}, we will use the large sieve inequality directly as it is stated above. However, we will later establish variational versions of this in Sections \ref{sec:VarBHD2} and \ref{sec:varLargeSieve}.

\subsection{A Variational Form of the Siegel-Walfisz Theorem}
For a positive real number $x$ and a Dirichlet character $\chi \mod q$, we define
\[\psi(x,\chi) = \sum_{p^{\alpha} \leq x} \chi(p^{\alpha}) \log(p), \text{   and   } \theta(x, \chi) = \sum_{p \leq x} \chi(p) \log(p).\]
For an interval $I$, we similarly define
\[\theta(I,\chi) = \sum_{ p \in I} \chi(p) \log(p).\]
We refer to the unique $\chi \mod q$ that takes the value 1 on all integers coprime to $q$ as the \emph{principal} character modulo $q$, and all other characters as non-principal.

The Siegel-Walfisz Theorem \cite{D} states:
\begin{theorem}\label{thm:standardSW}(Siegel-Walfisz Theorem)
Let $A$ be a positive real number. Then there exists some positive constant $c_A$ depending only on $A$ such that
\[|\psi(x,\chi)| \ll_A x e^{-c_A \log^{\frac{1}{2}} (x)}\]
for all non-principal characters $\chi \mod q$ for all moduli $q\leq \log^A(x)$.
\end{theorem}

We will find it more convenient to work with the following corollary:
\begin{corollary}\label{cor:SW}
Let $A$ be a positive real number. Then there exists some positive constant $c_A$ depending only on $A$ such that
\[|\theta(x,\chi)| \ll_A x e^{-c_A \log^{\frac{1}{2}} (x)}\]
for all non-principal characters $\chi \mod q$ for all moduli $q\leq \log^A(x)$.
\end{corollary}

\begin{proof} By the triangle inequality, $|\theta(x,\chi)| \leq |\psi(x,\chi)| + |\psi(x,\chi) - \theta(x,\chi)|$. The first quantity is bounded by Theorem \ref{thm:standardSW}. To bound the second quantity, we observe
\[|\psi(x,\chi) - \theta(x,\chi)| = \left| \sum_{\substack{p^{\alpha} \leq x \\ \alpha > 1}} \chi(p^{\alpha}) \log(p)\right| \leq \sum_{\substack{p^{\alpha} \leq x \\ \alpha>1}} \log(p) = \psi(x) - \theta(x),\]
by the triangle inequality and the fact that $|\chi(p^{\alpha})|$ is always either 0 or 1. Applying Lemma \ref{lem:basicBound}, we see that $|\psi(x,\chi) - \theta(x,\chi)| \ll x^{\frac{1}{2}}$, where the implicit constant is independent of $q$ and $\chi$.
\end{proof}

We now prove a variational form of this:
\begin{lemma}\label{lem:VarSW}
Let $A$ be a positive real number. Then there exists some positive constant $c'_A$ depending only on $A$ such that
\begin{equation}\label{eq:VarSW}
\sqrt{ \max_{\pi \in \mathcal{P}_x} \sum_{I \in \pi} \left|\theta(I,\chi)\right|^2} \ll_A x e^{-c'_A \log^{\frac{1}{2}}(x)}
\end{equation}
for all non-principal characters $\chi \mod q$ for all moduli $q\leq \log^A(x)$.
\end{lemma}

\begin{proof}
Since every $I \in \pi$ is a subinterval of $[1,x]$, the left hand side of (\ref{eq:VarSW}) is
\begin{equation}\label{eq:VarSW1}
\ll \sqrt{\max_{\pi \in \mathcal{P}_x} \sum_{I \in \pi} |\theta(I,\chi)|\cdot \max_{J \subseteq [1,x]} |\theta(J, \chi)|}.
\end{equation}

We consider the inner quantity $\sum_{I \in \pi} |\theta(I, \chi)|$. By definition, we have
\[\sum_{I \in \pi} |\theta(I,\chi)| = \sum_{I \in \pi} \left| \sum_{p \in I} \chi(p) \log(p)\right|.\]
Applying the triangle inequality, this is
\[\leq \sum_{I \in \pi} \sum_{p \in I} \left| \chi(p) \log(p)\right| \leq \sum_{I \in \pi} \sum_{ p \in I} \log(p) = \theta(x).\]
Here, we have used the fact that $|\chi(p)|$ is always either 1 or 0.

We then have that (\ref{eq:VarSW1}) is
\[\ll \sqrt{\theta(x) \max_{J \subseteq [1,x]} |\theta(J,\chi)|} .\]
Since $\theta(x) \leq \psi(x) \ll x$, this is
\begin{equation}\label{eq:VarSW2}
\ll \sqrt{x \max_{J \subseteq [1,x]} |\theta(J,\chi)|}.
\end{equation}

We consider the quantity $\max_{J \subseteq [1,x]} |\theta(J,\chi)|$. We observe that this is $\ll \max_{ y \leq x} |\theta(y,\chi)|$. We will upper bound this quantity for each $y$ separately. For larger $y$ values, we will employ Corollary \ref{cor:SW} for the value $2A$ (using $2A$ instead of $A$ will allow us to apply the corollary to a larger range of $y$ values). We let $c_{2A}$ denote the constant for $2A$ in the exponent. More precisely, for $y$ such that $q \leq \log^{2A}(y)$, we have $|\theta(y,\chi)| \ll_A y e^{-c_{2A} \log^{\frac{1}{2}}(y)}$ by Corollary \ref{cor:SW}. Since $y \leq x$, this is $\ll_A x e^{-c_{2A} \log^{\frac{1}{2}}(x)}$.

We now consider $y$ such that $\log^{2A}(y) \leq q$. This is equivalent to the condition $y \leq e^{q^{\frac{1}{2A}}}$. For these small $y$ values, we will use the basic estimate $|\theta(y,\chi)| \ll \theta(y) \ll y$. Since $\log^A(x)\geq q$ holds by assumption, we have $\log(x) \geq q^{\frac{1}{A}} \geq \log^2(y)$. We then have
$y \leq e^{\log^{\frac{1}{2}}(x)} \ll_A xe^{-c_{2A} \log^{\frac{1}{2}}(x)}.$
Hence,
\[\max_{ y \leq x} |\theta(y,\chi)| \ll_A xe^{-c_{2A}\log^{\frac{1}{2}}(x)}.\]

Thus, the quantity in (\ref{eq:VarSW2}) is \[\ll_A \sqrt{x^2 e^{-c_{2A} \log^{\frac{1}{2}}(x)}} = xe^{- \frac{1}{2}\cdot c_{2A} \log^{\frac{1}{2}}(x)}.\]
This proves Lemma \ref{lem:VarSW} with $c'_A := \frac{1}{2} \cdot c_{2A}$.
\end{proof}

We note that, conditional on the generalized Riemann hypothesis, for a nonprincipal Dirichlet of modulus $q$ one has the bound
\[|\psi(x, \chi)| \ll x^{1/2}\log(x) \log(qx)  \]
where the implied constant is absolute (see Theorem 13.7 in \cite{MV}). This can be used as in the argument above to obtain:

\begin{lemma}Let $\chi$ be a nonprincipal character mod $q$. Assuming the generalized Riemann hypothesis, we have that
\begin{equation}\label{eq:VarGRHSW1}
\max_{\pi \in \mathcal{P}_x} \sum_{I \in \pi} \left|\theta(I,\chi)\right|^2 \ll x^{3/2}\log(x) \log(qx).
\end{equation}
\end{lemma}

This could be used in place of Lemma \ref{lem:VarSW} in the following arguments to conditionally extend the range of $Q$ in the statements of Theorems \ref{thm:VarBHD} and \ref{thm:VarBHD2}. This is quite routine, and we omit the details. It may be possible to further improve (conditionally) the exponent of the $x^{3/2}$ term. We leave this as an interesting open problem.

\subsection{Proof of Theorem \ref{thm:VarBHD}}
We now bound the quantity
\begin{equation}\label{eq:VarBHD1}
\sum_{q \leq Q} \sum_{\substack{a \leq q \\ (a,q) = 1}} \max_{\pi \in \mathcal{P}_x} \sum_{I \in \pi} \left( \theta(I; q,a) - \frac{|I|}{\phi(q)}\right)^2.
\end{equation}
The structure of our proof will resemble the proof of the non-variational version of the theorem in \cite{D}.

We let $2^k$ denote the smallest power of two that is $\geq x$. We can then decompose $[1,2^k]$ into dyadic intervals $I_{c,\ell} = ((c-1)2^\ell, c2^{\ell}]$, where $\ell$ ranges from 0 to $2^k$ and $c$ ranges from $1$ to $2^{k-\ell}$. We note the following lemma \cite{LL}:
\begin{lemma}\label{lem:binarydecomp} Any subinterval of $S \subset [1,2^{k}]$ can be expressed as the disjoint union of intervals of the form $I_{c,\ell}$, such as
\[
S = \bigcup_{m} I_{c_{m},\ell_{m}}
\]
where at most two of the intervals $I_{c_{m},\ell_{m}}$ in the union are of each size, and where the union consists of at most $2k$ intervals.
\end{lemma}

In other words, each $I \subseteq [x]$ can be decomposed as a disjoint union of these dyadic intervals using at most two intervals on each level $\ell$. We let $D(I)$ denote the set of dyadic intervals in the decomposition of $I$. We observe \[\theta(I; q,a) - \frac{|I|}{\phi(q)} = \sum_{J \in D(I)}\left( \theta(J; q,a) - \frac{|J|}{\phi(q)}\right)\]
for any $I$, since $\sum_{J \in D(I)} |J| = |I|$. For each $\ell$, we let $D_\ell(I)$ denote the intervals in $D(I)$ on level $\ell$ (so $|D_\ell(I)| \leq 2$).
We can rewrite (\ref{eq:VarBHD1}) as:
\[\sum_{q \leq Q} \sum_{\substack{a \leq q \\ (a,q) = 1}} \max_{\pi \in \mathcal{P}_x} \sum_{I \in \pi} \left( \sum_{\ell =0}^{2^k} \sum_{J \in D_\ell(I)} \theta(J; q,a) - \frac{|J|}{\phi(q)} \right)^2.\]

By the triangle inequality for the $\ell^2$ norm, we have
\begin{equation}\label{eq:VarBHD2}
\sqrt{\sum_{q \leq Q} \sum_{\substack{a \leq q \\ (a,q) = 1}} \max_{\pi \in \mathcal{P}_x} \sum_{I \in \pi} \left( \theta(I; q,a) - \frac{|I|}{\phi(q)}\right)^2} \ll \sum_{\ell =0}^k \left( \sum_{q \leq Q} \sum_{\substack{a \leq q \\ (a,q) =1}} \max_{\pi \in \mathcal{P}_x} \sum_{I \in \pi} \left( \sum_{J \in D_\ell(I)} \theta(J;q,a) - \frac{|J|}{\phi(q)}\right)^2 \right)^{\frac{1}{2}}.
\end{equation}
Since $|D_\ell(I)| \leq 2$ for all $\ell, I$ and each dyadic interval can appear in $D(I)$ for at most one $I \in \pi$,
\[\max_{\pi \in \mathcal{P}_x} \sum_{I \in \pi} \left(\sum_{J \in D_\ell(I)} \theta(J; q,a) - \frac{|J|}{\phi(q)}\right)^2 \ll \sum_{c=1}^{2^{k-\ell}} \left( \theta(I_{c,\ell}; q,a) - \frac{|I_{c,\ell}|}{\phi(q)}\right)^2\]
for all $a,q, \ell$. Therefore the quantity in (\ref{eq:VarBHD2}) is
\begin{equation}\label{eq:VarBHD3}
\ll \sum_{\ell=0}^k \left( \sum_{q \leq Q} \sum_{\substack{a\leq q \\ (a,q) =1}} \sum_{c=1}^{2^{k-\ell}} \left(\theta(I_{c,\ell}; q,a)- \frac{|I_{c,\ell}|}{\phi(q)}\right)^2\right)^{\frac{1}{2}}.
\end{equation}

For each fixed $\ell$, we consider the quantity
\begin{equation}\label{eq:VarBHD4}
\sum_{q \leq Q} \sum_{\substack{a\leq q \\ (a,q) =1}} \sum_{c=1}^{2^{k-\ell}} \left(\theta(I_{c,\ell}; q,a)- \frac{|I_{c,\ell}|}{\phi(q)}\right)^2.
\end{equation}
We will bound this quantity using the large sieve inequality and the Siegel-Walfisz Theorem, so we need to first express it in terms of characters $\chi$ modulo $q$. For this, we will introduce some convenient notation. Fixing $q$, we let $\chi_0$ denote the principal character modulo $q$. For any character $\chi$ modulo $q$, we define
\[\theta'(I, \chi) := \left\{
                        \begin{array}{ll}
                          \theta(I, \chi), & \hbox{$\chi \neq \chi_0$;} \\
                          \theta(I, \chi_0) - |I|, & \hbox{$\chi = \chi_0$.}
                        \end{array}
                      \right.
\]

We will employ the following lemma:
\begin{lemma}\label{lem:identity} For any interval $I$ and any coprime positive integers $q,a$,
\[\theta(I; q,a) - \frac{|I|}{\phi(q)} = \frac{1}{\phi(q)} \sum_{\chi \mod q} \overline{\chi}(a) \theta'(I,\chi),\]
where $\overline{\chi}$ denotes the character obtained from $\chi$ by complex conjugation.
\end{lemma}

\begin{proof} We note that for each integer $n$,
\[\sum_{\chi \mod q} \chi(n) = \left\{
                                 \begin{array}{ll}
                                   \phi(q), & \hbox{ $n \equiv 1 \mod q$;} \\
                                   0, & \hbox{otherwise.}
                                 \end{array}
                               \right.\]
For any integer $a$ coprime to $q$, we let $\overline{a}$ denote an integer such that $\overline{a} a \equiv 1 \mod q$. Then, for any $\chi \mod q$, $\overline{\chi}(a) \chi(n) = \chi(\overline{a}) \chi(n) = \chi(\overline{a} n)$. Since $\overline{a} n \equiv 1 \mod q$ if and only if $n \equiv a \mod q$, we have
\[\frac{1}{\phi(q)} \sum_{\chi \mod q} \overline{\chi}(a) \chi(n) = \left\{
                                                                      \begin{array}{ll}
                                                                        1, & \hbox{$ n \equiv a \mod q$;} \\
                                                                        0, & \hbox{otherwise.}
                                                                      \end{array}
                                                                    \right.\]

We then observe
\begin{eqnarray*}
  \theta(I; q, a) &=& \sum_{\substack{p \in I\\ p \equiv a \mod q}} \log (p)  \\
   &=& \frac{1}{\phi(q)} \sum_{p \in I} \log (p) \sum_{ \chi \mod q} \overline{\chi}(a) \chi(p) \\
   &=& \frac{1}{\phi(q)} \sum_{\chi \mod q} \overline{\chi}(a) \theta(I, \chi).
\end{eqnarray*}
By definition of $\theta'(I,\chi)$, it then follows that
\[\theta(I; q,a) - \frac{|I|}{\phi(q)} = \frac{1}{\phi(q)} \sum_{\chi \mod q} \overline{\chi}(a) \theta'(I,\chi).\]
\end{proof}

The quantity (\ref{eq:VarBHD4}) can then be expressed as
\begin{equation}\label{eq:VarBHD5}
\sum_{q \leq Q} \frac{1}{\phi(q)^2}\sum_{c=1}^{2^{k-\ell}} \sum_{\substack{a \leq q \\ (a,q) =1}} \left| \sum_{\chi \mod q} \overline{\chi}(a) \theta'(I_{c,\ell},\chi)\right|^2.
\end{equation}

For fixed $q$ and $c$, the inner quantity can be expanded as:
\[ = \sum_{\substack{a \leq q \\ (a,q) =1}} \sum_{\chi_1 \mod q} \; \sum_{\chi_2 \mod q} \overline{\chi_1}(a) \chi_2(a) \theta'(I_{c,\ell},\chi_1) \overline{\theta'(I_{c,\ell},\chi_2)}\]
Reordering the sums, this is
\[\sum_{\chi_1 \mod q} \; \sum_{\chi_2 \mod q} \theta'(I_{c,\ell},\chi_1) \overline{\theta'(I_{c,\ell},\chi_2)} \sum_{\substack{a \leq q \\ (a,q) =1}} \overline{\chi_1}(a) \chi_2(a).\]
The innermost sum is now the inner product of the characters $\chi_1, \chi_2$. Since the distinct characters modulo $q$ are orthogonal under this inner product, this innermost sum is 0 unless $\chi_1 = \chi_2$. Therefore, (\ref{eq:VarBHD5}) is equal to
\begin{equation}\label{eq:VarBHD6}
\sum_{q \leq Q} \frac{1}{\phi(q)} \sum_{c =1}^{2^{k-\ell}} \sum_{\chi \mod q} \left| \theta'(I_{c,\ell},\chi)\right|^2.
\end{equation}

In order to use the large sieve inequality as stated in Theorem \ref{thm:standardLargeSieve}, we need to adjust our character sum to be over the primitive characters modulo $q$ instead of all characters modulo $q$. For this, we first note that every character $\chi$ modulo $q$ is induced by some primitive character $\chi_1$ modulo $q_1$ where $q_1 \leq q$. We then have:

\begin{lemma}\label{lem:primitive}
For any $I$ and any character $\chi$ modulo $q$ induced by $\chi_1$ modulo $q_1$,
\[ \left| \theta'(I,\chi_1) - \theta'(I,\chi)\right| \leq \log(q).\]
\end{lemma}

\begin{proof} For all integers $n$ coprime to $q$, $\chi(n) = \chi_1(n)$. In fact,
\[ \theta'(I,\chi_1) - \theta'(I,\chi) = \sum_{\substack{p \in I \\ p|q}} \chi_1(p) \log(p).\]
Therefore,
\[\left| \theta'(I,\chi_1) - \theta'(I,\chi)\right| \leq \sum_{\substack{p \in I \\ p|q}} \log(p) \leq \log(q).\]
To see the final inequality, consider the prime factorization of $q = p_1^{\alpha_1} \cdots p_r^{\alpha_r}$. Then $\log (q) = \alpha_1 \log(p_1) + \cdots + \alpha_r \log(p_r)$.
\end{proof}

As a consequence of Lemma \ref{lem:primitive}, we have $|\theta'(I_{c,\ell},\chi)|^2 \ll |\theta'(I_{c,\ell},\chi_1)|^2 + \log^2 q$ for all dyadic intervals $I_{c,\ell}$ and all non-primitive characters $\chi$. Thus, the quantity in (\ref{eq:VarBHD6}) is
\begin{eqnarray}\label{eq:VarBHD7}\nonumber
 & \ll & \sum_{q \leq Q} \frac{1}{\phi(q)} \sum_{c=1}^{2^{k-\ell}} \sum_{\chi \mod q} \left( \log^2(q) + |\theta'(I_{c,\ell},\chi_1)|^2\right) \\
 &= &\sum_{q \leq Q} \frac{1}{\phi(q)} \sum_{c=1}^{2^{k-\ell}} \sum_{\chi \mod q} \log^2(q) + \sum_{q \leq Q} \frac{1}{\phi(q)}\sum_{c=1}^{2^{k-\ell}} \sum_{\chi \mod q} |\theta'(I_{c,\ell},\chi_1)|^2.
\end{eqnarray}
As above, $\chi_1$ here denotes the primitive character that induces $\chi$.

We bound the contribution of this first sum to the quantity in (\ref{eq:VarBHD3}) (noting that there are $\phi(q)$ characters modulo $q$):
\[\sum_{0 \leq \ell \leq k} \sqrt{\sum_{q \leq Q} 2^{k-\ell} \log^2(q)} = \left( \sum_{0 \leq \ell \leq k} (2^{\frac{1}{2}})^{k-\ell} \right) \sqrt{\sum_{q \leq Q} \log^2(q)} \ll 2^{\frac{k}{2}} \sqrt{Q \log^2(Q)}.\]
Since (\ref{eq:VarBHD3}) is an upper bound on the square root of (\ref{eq:VarBHD1}), the contribution to (\ref{eq:VarBHD1}) is therefore
$\ll 2^k Q \log^2(Q) \ll x Q\log^2(x)$,
which is acceptable.

It thus suffices to consider
\begin{equation}\label{eq:VarBHD8}
\sum_{q \leq Q} \frac{1}{\phi(q)} \sum_{c=1}^{2^{k-\ell}} \sum_{\chi \mod q} |\theta'(I_{c,\ell}, \chi_1)|^2
\end{equation}
for each fixed $\ell$.
Each primitive character $\chi_1$ modulo $q_1$ induces characters $\chi$ modulo $q$ for every $q$ that is a multiple of $q_1$. We can use this to rewrite the above quantity in terms of a sum over only primitive characters:
\[= \sum_{q \leq Q} \sum_{c=1}^{2^{k-\ell}} \; \sideset{}{^*} \sum_{\chi \mod q} |\theta'(I_{c,\ell},\chi)|^2 \left( \sum_{j \leq \frac{Q}{q}} \frac{1}{\phi(jq)}\right).\]
We note that
\[\sum_{j \leq \frac{Q}{q}} \frac{1}{\phi(jq)} \ll \phi(q)^{-1} \log(2Q/q)\]
(see \cite{D}, pp. 163), so this is
\begin{equation}\label{eq:VarBHD9}
\ll \sum_{q \leq Q} \frac{1}{\phi(q)} \log\left(\frac{2Q}{q}\right) \sum_{c=1}^{2^{k-\ell}} \; \sideset{}{^*} \sum_{\chi \mod q} |\theta'(I_{c,\ell},\chi)|^2.
\end{equation}

We will split this sum over $q \leq Q$ into ranges and bound each piece separately. For a fixed $1\leq U \leq Q$, we consider
\begin{equation}\label{eq:VarBHD10}
\sum_{c=1}^{2^{k-\ell}} \sum_{ U < q \leq 2U} \frac{1}{\phi(q)} \log \left(\frac{2Q}{q}\right) \; \sideset{}{^*}\sum_{ \chi \mod q} |\theta(I_{c,\ell}, \chi)|^2.
\end{equation}
Note that we have switched notation from $\theta'$ to $\theta$ here without changing the quantity, since $\theta'$ and $\theta$ only differ on the trivial character, and this is only included in the primitive characters modulo $q$ when $q =1$. Since $U\geq 1$ and our sum here is over $q > U$, $\theta$ and $\theta'$ behave identically here.

We observe that for each fixed $c$, the contribution to (\ref{eq:VarBHD10}) is
\[ \ll U^{-1} \log\left(\frac{2Q}{U}\right) \sum_{ q \leq 2U} \frac{q}{\phi(q)} \; \sideset{}{^*} \sum_{\chi \mod q} |\theta(I_{c,\ell}, \chi)|^2. \]
Letting $a_p := \log(p)$ for primes $p$ and $a_n := 0$ for all non-primes $n$, we apply Theorem \ref{thm:standardLargeSieve} to see that this is
\[
\ll U^{-1} \log\left(\frac{2Q}{U}\right)\left(|I_{c,\ell}| + U^2\right) \sum_{p \in I_{c,\ell}} \log^2(p).
\]

We define $Q_1 := \log^{A+1}(x)$. We consider the values $U = Q2^{-j}$ as $j$ ranges from 0 to $J := \lceil \log(\frac{Q}{Q_1}) \rceil$. We then have:
\[\sum_{Q_1 < q \leq Q} \frac{1}{\phi(q)} \log\left(\frac{2Q}{q}\right) \; \sideset{}{^*} \sum_{\chi \mod q} |\theta(I_{c,\ell}, \chi)|^2 \ll Q^{-1} \sum_{ j=0}^{J} 2^j (j+1) \left(|I_{c,\ell}| + Q^2 2^{-2j}\right) \sum_{p \in I_{c,\ell}} \log^2(p).  \]
We expand the latter quantity as
\[
 = |I_{c,\ell}| Q^{-1} \left(\sum_{j=0}^{J} (j+1)2^j\right) \left(\sum_{p \in I_{c,\ell}} \log^2(p)\right) + Q \left( \sum_{j=0}^{J} (j+1) 2^{-j}\right) \left(\sum_{p \in I_{c,\ell}} \log^2(p)\right).
\]
Inserting this into the sum over the $c$ values, we then have
\[
\sum_{c=1}^{2^{k-\ell}} \sum_{Q_1 < q \leq Q} \frac{1}{\phi(q)}\log\left(\frac{2Q}{q}\right) \sideset{}{^*}\sum_{\chi \mod q} |\theta(I_{c,\ell},\chi)|^2 \] \[\ll \sum_{c =1}^{2^{k-\ell}} |I_{c,\ell}| Q^{-1} \left(\sum_{j=0}^{J} (j+1)2^j\right) \left(\sum_{p \in I_{c,\ell}} \log^2(p)\right) + \sum_{c=1}^{2^{k-\ell}} Q \left( \sum_{j=0}^{J} j 2^{-j}\right) \left(\sum_{p \in I_{c,\ell}} \log^2(p)\right).
\]

We observe that
\[\sum_{c =1}^{2^{k-\ell}} \sum_{p \in I_{c,\ell}} \log^2(p)  = \sum_{ p \leq 2^k} \log^2(p),\] since the union of the intervals $I_{c,\ell}$ as $c$ ranges from 1 to $2^{k-\ell}$ is equal to the interval $[2^k]$.
We then have that $\sum_{p \leq 2^k} \log^2(p) \ll k 2^k$.
We also note that
\[\sum_{j=0}^J (j+1)2^j = J2^{J+1} + 1 \ll J 2^J \text{  and  } \sum_{j=0}^J (j+1) 2^{-j} \ll 1.\]

Therefore, we have shown
\begin{equation}\label{eq:VarBHD11}
\sum_{c=1}^{2^{k-\ell}} \sum_{Q_1 < q \leq Q} \frac{1}{\phi(q)}\log\left(\frac{2Q}{q}\right) \sideset{}{^*}\sum_{\chi \mod q} |\theta(I_{c,\ell},\chi)|^2 \ll Q^{-1} ( 2^{\ell})( J 2^{J})( k 2^k) + Q ( k 2^k).
\end{equation}
Recalling that $2^k \ll x$, $k \ll \log(x)$, $J = \lceil \log(\frac{Q}{Q_1}) \rceil$, and $Q \leq x$, we see that the contribution to (\ref{eq:VarBHD9}) from $q$'s between $Q_1$ and $Q$ is
\[ \ll Q_1^{-1} (2^\ell)(x \log^2(x))+ Q (x \log(x)).\]


We now consider values of $q \leq Q_1$. For every primitive character $\chi$ modulo $q$ where $q >1$, $\chi$ is non-principal. Note that for the principal character $\chi_0$ modulo 1, $\theta'(I_{c,\ell},\chi_0) = \theta(I_{c,\ell}, \chi_0) - |I_{c,\ell}| = 0$. Thus, the contribution to (\ref{eq:VarBHD9}) from values $q \leq Q_1$ can be written as:
\begin{equation}\label{eq:VarBHD12}
\sum_{ 1 < q \leq Q_1} \frac{1}{\phi(q)} \log\left(\frac{2Q}{q}\right) \sideset{}{^*}\sum_{\chi \mod q} \sum_{c=1}^{2^{k-\ell}} |\theta'(I_{c,\ell},\chi)|^2.
\end{equation}
This innermost sum over $c$ is a sum over a partition of $[2^k]$, so we can apply Lemma \ref{lem:VarSW} (with $A+1$ as the constant) to conclude that
\[\sum_{c=1}^{2^{k-\ell}} |\theta'(I_{c,\ell},\chi)|^2 \ll_A x^2 e^{-\tilde{c}_A \log^{\frac{1}{2}}(x)}\]
for some positive constant $\tilde{c}_A$ depending only on $A$. The quantity in (\ref{eq:VarBHD12}) is then \\ $\ll_A Q_1 \log(Q)\cdot x^2 e^{-\tilde{c}_A \log^{\frac{1}{2}}(x)}$.

Putting this all together, we have that the quantity in (\ref{eq:VarBHD9}) is:
\[ \ll_A  Q_1 \log(Q)\cdot x^2 e^{-\tilde{c}_A \log^{\frac{1}{2}}(x)} +  Q_1^{-1} (2^\ell)(x \log^2(x))+ Q (x \log(x)).\]
Thus, the contribution to (\ref{eq:VarBHD3}) is bounded by:
\[\ll_A \sum_{\ell=0}^k \sqrt{Q_1 \log(Q)\cdot x^2 e^{-\tilde{c}_A \log^{\frac{1}{2}}(x)} + Q_1^{-1} (2^\ell)(x \log^2(x))+ Q (x \log(x))} \]
\[\ll_A \log(x) \sqrt{Q_1 \log(Q) \cdot x^2 e^{-\tilde{c}_A \log^{\frac{1}{2}}(x)}} + \sqrt{Q_1^{-1} x^2 \log^2(x)} + \log(x) \sqrt{Q(x\log(x))}.\]
Hence the contribution to (\ref{eq:VarBHD1}) is bounded by the square of this:
\[ \ll_A Q_1 \log(Q) \log^2(x) x^2 e^{-\tilde{c}_A \log^{\frac{1}{2}}(x)} + Q_1^{-1} x^2 \log^2(x) + Q x \log^3(x).\]

Recalling that $Q_1 = \log^{A+1}(x)$ and $x\log^{-A}(x) \leq Q \leq x$, we see that this is
\[ \ll_A \log^{A+4}(x)x^2 e^{-\tilde{c}_A \log^{\frac{1}{2}}(x)} + Qx \log(x) + Qx \log^3(x).\]
Since $e^{-\tilde{c}_A \log^{\frac{1}{2}}(x)} \ll_A \log^{-2A-1}(x)$, this first term is $ \ll_A Qx\log^3(x)$, as required. This completes the proof of Theorem \ref{thm:VarBHD}.

\subsection{An Averaged Variant of Erd\H{o}s' Conjecture}\label{sec:erdos}
We now apply our variational form of the Barban-Davenport-Halberstam Theorem to prove Corollary \ref{cor:erdos}.

\begin{cor8}
Let $A > 0$. For all positive real numbers $x$ and $Q$ satisfying $x (\log(x))^{-A} \leq Q \leq x$,
\[\sum_{q \leq Q} \sum_{\substack{a \leq q \\ (a,q)=1}} \sum_{p_{i+1}^{a,q} \leq x} \left( \frac{p_{i+1}^{a,q}-p_i^{a,q}}{\phi(q)}\right)^2 \ll Qx \log^3(x).\]
\end{cor8}

\begin{proof} For each fixed $a,q$, we consider a partition in $\mathcal{P}_x$ containing all the intervals of the form
\[I_i := \left(p_i^{a,q}, p_{i+1}^{a,q}\right).\]
By definition of $\theta$ and $I_i$, the quantity $\theta(I_i; q,a)$ equals 0 for all $i$. Hence,
\[\left(\theta(I_i; q,a) - \frac{|I_i|}{\phi(q)}\right)^2 =  \left(\frac{p_{i+1}^{a,q} - p_i^{a,q}-1}{\phi(q)}\right)^2.\]
We note that $(p_{i+1}^{a,q}-p_i^{a,q}-1)^2 \gg (p_{i+1}^{a,q} - p_i^{a,q})^2$ (except for the case where $p_{i+1}^{a,q} = 3$ and $p_i^{a,q} = 2$, but this only occurs for $q = 1$ and so can be ignored). Thus, Theorem \ref{thm:VarBHD} implies
\[\sum_{q \leq Q} \sum_{\substack{a \leq q \\ (a,q)=1}} \sum_{p_{i+1}^{a,q} \leq x} \left( \frac{p_{i+1}^{a,q}-p_i^{a,q}}{\phi(q)}\right)^2 \ll Qx \log^3(x).\]
\end{proof}

\section{Another Variational Form of the Barban-Davenport-Halberstam Theorem}\label{sec:VarBHD2}
We now prove:

\begin{thm5} Let $A > 0$. For all positive real numbers $x$ and $Q$ satisfying $x (\log(x))^{-A} \leq Q \leq x$,
\[\sum_{q \leq Q} \max_{ \pi \in \mathcal{P}_x} \sum_{\substack{a\leq q \\ (a,q)=1}} \sum_{I \in \pi} \left(\theta(I;a,q)-\frac{|I|}{\phi(q)}\right)^2 \ll_A xQ \log^2(x).\]
\end{thm5}


We will need some additional notation. We let $e(x) := e^{2\pi i x}$. For any real numbers $\{a_n\}_{n=1}^N$, we define a function $S:\mathbb{T} \rightarrow \mathbb{C}$ by
\[S(\alpha) := \sum_{n=1}^N a_n e(n\alpha).\]
For $\delta >0$, we say that points $\alpha_1, \ldots, \alpha_R \in \mathbb{T}$ are \emph{$\delta$-separated} if $||\alpha_i - \alpha_j|| \geq \delta$ for all $i\neq j$, where the $||\cdot ||$ denotes the norm modulo 1.

We let $\mathcal{P}_N$ denote the set of all partitions of $[N]$ into a union of disjoint intervals.
We then define
\[||S(\alpha)||_{V^r} := \max_{ \pi \in \mathcal{P}_N}\left( \sum_{I \in \pi} \left| \sum_{n\in I} a_n e(n\alpha)\right|^r\right)^{\frac{1}{r}}.\]

We note the variational Carleson Theorem \cite{OSTTW}:
\begin{theorem}\label{thm:VarCar}
For any real numbers $\{a_n\}_{n=1}^N$ and any $r>2$,
\[\int_{\mathbb{T}} ||S(\alpha)||^2_{V^r} d\alpha \ll_r \sum_{n=1}^N |a_n|^2.\]
\end{theorem}

The case of $r=2$ is addressed in the following theorem, which follows immediately from Corollary 4 in \cite{LL}:
\begin{theorem}\label{thm:r2}
For any real numbers $\{a_n\}_{n=1}^N$,
\[\int_{\mathbb{T}} ||S(\alpha)||^2_{V^2} d\alpha \ll \log(N) \sum_{n=1}^N |a_n|^2.\]
\end{theorem}
\begin{flushleft}
We note that the $\log(N)$ factor is known to be sharp.
We will first prove the following lemma, which is a variational version of the analytic large sieve inequality.
\end{flushleft}

\begin{lemma}\label{lem:analyticLS}
For any $\delta >0$ and for any points $\alpha_1, \ldots, \alpha_R \in \mathbb{T}$ that are $\delta$-separated,
\[ \sum_{i=1}^R ||S(\alpha_i)||^2_{V^r} \ll_r (N+\delta^{-1} +1 )\sum_{n=1}^N |a_n|^2\]
for any $r>2$. Also,
\[\sum_{i=1}^R ||S(\alpha_i)||^2_{V^2} \ll (N+\delta^{-1} +1 )\log(N) \sum_{n=1}^N |a_n|^2.\]
\end{lemma}

\begin{proof}
This proof will be a variational adaptation of the proof of Theorem 5 in \cite{MVaa}.
By a theorem of Selberg \cite{V}, there exists an entire function $K(z)$ such that $K$ is real-valued on $\mathbb{R}$, $K(x) \geq 0$ for all real $x$, and $K(x) \geq 1$ for all $1 \leq x\leq N$. Moreover, $K(x)$ is integrable, and $\widehat{K}(0) = N-1 + \delta^{-1}$. By a theorem of Fej\'{e}r, there is another entire function $k(z)$ such that $K(x) = |k(x)|^2$ for all $x \in \mathbb{R}$, and $\hat{k}$ (the fourier transform of $k(x)$) has support in $\left( -\frac{\delta}{2}, \frac{\delta}{2}\right)$. We note that
\[k(x) = \int_{-\frac{\delta}{2}}^{\frac{\delta}{2}} \hat{k}(\xi) e(x\xi) d\xi.\]

We define a function $T:\mathbb{T} \rightarrow \mathbb{C}$ by
\[T(\alpha) := \sum_{n=1}^N a_n k(n)^{-1} e(n\alpha).\]
Similarly, we define
\[||T(\alpha)||_{V^r} := \max_{ \pi \in \mathcal{P}_N}\left( \sum_{I \in \pi} \left| \sum_{n\in I} a_n k(n)^{-1} e(n\alpha)\right|^r\right)^{\frac{1}{r}}.\]

For any $\alpha$ and any $r \geq 2$, we have
\begin{eqnarray*}
||S(\alpha)||_{V^r} & = & \max_{ \pi \in \mathcal{P}_N}\left( \sum_{I \in \pi} \left| \sum_{n\in I} a_n e(n\alpha)\right|^r\right)^{\frac{1}{r}} \\
 & = & \max_{\pi \in \mathcal{P}_N} \left( \sum_{I \in \pi} \left| \sum_{n \in I} a_n k(n)^{-1} e(n\alpha) \int_{-\frac{\delta}{2}}^{\frac{\delta}{2}} \hat{k}(\xi) e(n\xi) d\xi \right|^r\right)^{\frac{1}{r}}\\
 & = & \max_{\pi \in \mathcal{P}_N} \left( \sum_{I \in \pi} \left| \int_{-\frac{\delta}{2}}^{\frac{\delta}{2}} \hat{k}(\xi) \sum_{n\in I} a_n k(n)^{-1} e(n(\alpha + \xi))d\xi \right|^r\right)^{\frac{1}{r}}.
\end{eqnarray*}

By Minkowski's integral inequality (see \cite{HLP}, Theorem 201 for example), this last quantity is
\[ \leq \max_{\pi \in \mathcal{P}_N} \int_{-\frac{\delta}{2}}^{\frac{\delta}{2}} \left( \sum_{I \in \pi} \left| \hat{k}(\xi) \sum_{n\in I} a_n k(n)^{-1} e(n(\alpha + \xi))\right|^r\right)^{\frac{1}{r}} d \xi\]
\[ = \max_{\pi \in \mathcal{P}_N} \int_{-\frac{\delta}{2}}^{\frac{\delta}{2}} |\hat{k}(\xi)| \left( \sum_{I \in \pi} \left|\sum_{n\in I} a_n k(n)^{-1} e(n(\alpha + \xi))\right|^r\right)^{\frac{1}{r}} d \xi\]
\[\leq \int_{-\frac{\delta}{2}}^{\frac{\delta}{2}} |\hat{k}(\xi)| \max_{\pi \in \mathcal{P}_N} \left( \sum_{I \in \pi} \left|\sum_{n\in I} a_n k(n)^{-1} e(n(\alpha + \xi))\right|^r\right)^{\frac{1}{r}} d \xi.\]

Now applying the Cauchy-Schwarz inequality, we see this is
\[\leq \left( \int_{-\frac{\delta}{2}}^{\frac{\delta}{2}} |\hat{k}(\xi)|^2 d\xi\right)^{\frac{1}{2}} \left( \int_{-\frac{\delta}{2}}^{\frac{\delta}{2}} ||T(\xi+\alpha)||^2_{V^r} d\xi\right)^{\frac{1}{2}}.\]
By the properties of $k$ and $K$, we have
\[\left( \int_{-\frac{\delta}{2}}^{\frac{\delta}{2}} |\hat{k}(\xi)|^2 d\xi\right)^{\frac{1}{2}} = \left(\int_{-\infty}^\infty |k(x)|^2 dx\right)^{\frac{1}{2}} = \left( \int_{-\infty}^\infty K(x)dx \right)^{\frac{1}{2}} = \left(\widehat{K}(0)\right)^{\frac{1}{2}} = \left( N-1 + \delta^{-1}\right)^{\frac{1}{2}}.\]

Therefore, we have shown that
\[\sum_{i=1}^R ||S(\alpha_i)||^2_{V^r} \leq (N-1+\delta^{-1}) \sum_{i=1}^R \int_{-\frac{\delta}{2}}^{\frac{\delta}{2}} ||T(\xi+\alpha_i)||^2_{V^r} d\xi.\]
Since the $\alpha_i$'s are $\delta$-separated, the ranges $(-\frac{\delta}{2}+ \alpha_i, \frac{\delta}{2} + \alpha_i)$ are disjoint in $\mathbb{T}$, and hence
\begin{equation}\label{eq:int}
\sum_{i=1}^R ||S(\alpha_i)||^2_{V^r} \leq (N-1+\delta^{-1}) \int_{\mathbb{T}} ||T(\xi)||^2_{V^r} d\xi.
\end{equation}

For $r>2$, we may apply Theorem \ref{thm:VarCar} for the real numbers $\{a_n k(n)^{-1}\}_{n=1}^N$ to conclude that the righthand side of (\ref{eq:int}) is
\[ \ll_r (N-1+\delta^{-1})\sum_{n=1}^N |a_n k(n)^{-1}|^2.\]
Recalling that $\frac{1}{|k(n)|^2} = \frac{1}{K(n)}$ and $K(n) \geq 1$ for all $n$ from 1 to $N$, we obtain
\[\sum_{i=1}^R ||S(\alpha_i)||^2_{V^r} \ll_r (N-1+\delta^{-1}) \sum_{n=1}^N |a_n|^2\]
for all $r>2$, as required.

For $r=2$, we may apply Theorem \ref{thm:r2} for the real numbers $\{a_n k(n)^{-1}\}_{n=1}^N$ to conclude that the righthand side of (\ref{eq:int}) is
\[ \ll (N-1+\delta^{-1}) \log(N) \sum_{n=1}^N |a_n k(n)^{-1}|^2 \ll (N-1+\delta^{-1})\log(N) \sum_{n=1}^N |a_n|^2.\]
\end{proof}

We next prove the following lemma:
\begin{lemma}\label{lem:mid}
For any real numbers $\{a_n\}_{n=1}^N$,
\[\sum_{q \leq Q} \frac{q}{\phi(q)} \max_{\pi \in \mathcal{P}_N} \sideset{}{^*}\sum_{\chi \mod q} \sum_{I \in \pi} |T(I,\chi)|^2 \ll (N+Q^2)\log(N) \sum_{n=1}^N |a_n|^2.\]
\end{lemma}

\begin{proof}
This proof is adapted from the proof of Theorem 4 in \cite{D}. For a character $\chi$ modulo $q$, we define the complex value $\tau(\chi)$ by
\[\tau(\chi) := \sum_{a\leq q} \chi(a) e\left(\frac{a}{q}\right).\]
We have
\[\chi(n) = \frac{1}{\tau(\overline{\chi})}\sum_{a \leq q} \overline{\chi}(a) e\left(\frac{an}{q}\right)\]
for all $n$ for all primitive $\chi$ (see \cite{D}, chapter 9).
Therefore, for any interval $I$ and any primitive $\chi$ modulo $q$,
\[ T(I,\chi) = \sum_{n\in I} a_n \chi(n) = \frac{1}{\tau(\overline{\chi})} \sum_{a \leq q} \overline{\chi}(a) \sum_{n\in I} a_n e\left(\frac{an}{q}\right).\]

We then note that when $\chi$ is primitive, $|\tau(\chi)| = q^{\frac{1}{2}}$ (see \cite{D}, p. 66). This yields (for any partition $\pi$):
\[
\sum_{I\in \pi} |T(I,\chi)|^2 = \frac{1}{q} \sum_{I \in \pi} \left| \sum_{a\leq q} \overline{\chi}(a) \sum_{n \in I} a_n e\left(\frac{an}{q}\right)\right|^2.
\]
Now summing over primitive characters, we have
\[
\sideset{}{^*}\sum_{\chi \mod q} \; \sum_{I \in \pi} |T(\chi,I)|^2 = \frac{1}{q} \sum_{I \in \pi} \; \sideset{}{^*}\sum_{\chi\mod q} \left|\sum_{a \leq q} \overline{\chi}(a) \sum_{n\in I} a_n e\left(\frac{an}{q}\right)\right|^2.
\]

This quantity can only increase if we sum over all characters modulo $q$, so this is
\begin{equation}\label{eq:allchars}
\leq \frac{1}{q} \sum_{I \in \pi} \sum_{a \leq q} \sum_{b \leq q} \left( \sum_{n \in I} a_n e\left(\frac{an}{q}\right)\right) \left(\sum_{n\in I} a_n \overline{e}\left(\frac{bn}{q}\right)\right) \sum_{\chi \mod q} \overline{\chi}(a)\chi(b).
\end{equation}
Here, $\overline{e}$ denotes conjugation. We note that this innermost sum over characters is equal to 0 unless $a = b$ and $a$ is coprime to $q$. In this case, it is equal to $\phi(q)$. Plugging this into (\ref{eq:allchars}), we have
\[ = \frac{\phi(q)}{q} \sum_{I \in \pi} \sum_{\substack{a \leq q \\ (a,q) =1}} \left| \sum_{n \in I} a_n e\left(\frac{an}{q}\right)\right|^2.\]

Recalling that
\[||S(\alpha)||^2_{V^2} := \max_{\pi \in \mathcal{P}_N} \sum_{I \in \pi} \left| \sum_{n \in I} a_n e(n\alpha)\right|^2,\]
we conclude that
\begin{eqnarray*}
\sum_{q \leq Q} \frac{q}{\phi(q)} \max_{\pi \in \mathcal{P}_N} \sideset{}{^*}\sum_{\chi \mod q} \sum_{I \in \pi} |T(I,\chi)|^2 & \ll &  \sum_{q \leq Q} \max_{\pi \in \mathcal{P}_N} \sum_{I \in \pi} \sum_{\substack{a \leq q \\ (a,q) =1}} \left| \sum_{n\in I} a_n e\left(\frac{an}{q}\right)\right|^2\\
 & \leq & \sum_{q \leq Q} \sum_{\substack{a \leq q \\ (a,q)=1}} ||S\left(a/q\right)||^2_{V^2}.
\end{eqnarray*}
Here we have used the fact that moving the maximum inside the sum over $a$'s coprime to $q$ can only make the quantity larger. The points $\frac{a}{q}$ as $q$ ranges from 1 to $Q$ and $a$ ranges over values coprime to $q$ are $\frac{1}{Q^2}$-separated as points in $\mathbb{T}$. Thus, applying Lemma \ref{lem:analyticLS}, this is $\ll (N + Q^2) \log(N)\sum_{n=1}^N |a_n|^2$.
\end{proof}

We are now equipped to prove Theorem \ref{thm:VarBHD2}.
We recall Lemma \ref{lem:identity}, which states that
\[\theta(I;q,a) - \frac{|I|}{\phi(q)} = \frac{1}{\phi(q)}\sum_{\chi \mod q} \overline{\chi}(a) \theta'(I,\chi).\]
We then have
\[
\sum_{q\leq Q} \max_{\pi \in \mathcal{P}_x} \sum_{\substack{a \leq q\\ (a,q)=1}} \; \sum_{I \in \pi} \left(\theta(I;q,a)-\frac{|I|}{\phi(q)}\right)^2  = \sum_{q \leq Q}\frac{1}{\phi(q)^2} \max_{\pi \in \mathcal{P}_x} \sum_{\substack{a\leq q \\ (a,q)=1}} \; \sum_{I \in \pi} \left| \sum_{\chi \mod q} \overline{\chi}(a) \theta'(I, \chi)\right|^2.
\]

By expanding the square and rearranging the sums inside the maximum, this is
\[ = \sum_{q \leq Q} \frac{1}{\phi(q)^2} \max_{\pi \in \mathcal{P}_x} \sum_{I \in \pi}\; \sum_{\chi_1 \mod q}\;  \sum_{\chi_2 \mod q} \theta'(I, \chi_1) \overline{\theta'(I,\chi_2)} \sum_{\substack{a\leq q\\ (a,q)=1}} \overline{\chi_1}(a) \chi_2(a).\]
This innermost sum over $a$'s coprime to $q$ is equal to $\phi(q)$ whenever $\chi_1 = \chi_2$, and is equal to 0 otherwise. Hence this quantity is
\begin{equation}\label{eq:BHD2-1}
 = \sum_{q \leq Q} \frac{1}{\phi(q)} \max_{\pi \in \mathcal{P}_x} \sum_{I \in \pi} \sum_{\chi \mod q} |\theta'(I,\chi)|^2.
\end{equation}

Each character $\chi \mod q$ is induced by some primitive character $\chi_1 \mod q_1$, where $q_1$ divides $q$. Applying the triangle inequality, we see that (\ref{eq:BHD2-1}) is
\begin{equation}\label{eq:BHD2-2}
\ll \sum_{q\leq Q} \frac{1}{\phi(q)} \max_{\pi \in \mathcal{P}_x} \sum_{I \in \pi} \sum_{\chi \mod q} |\theta'(I, \chi_1)|^2 + \sum_{q \leq Q} \frac{1}{\phi(q)} \max_{\pi \in \mathcal{P}_x} \sum_{I \in \pi} \sum_{\chi \mod q} |\theta'(I,\chi)-\theta'(I,\chi_1)|^2.
\end{equation}

We consider the second quantity in (\ref{eq:BHD2-2}). This is
\[ \leq \sum_{q \leq Q}\frac{1}{\phi(q)} \max_{\pi \in \mathcal{P}_x} \left( \sum_{I \in \pi} \sum_{\chi \mod q} |\theta'(I,\chi)-\theta'(I,\chi_1)|\right)^2.\]
Recalling the definitions of $\theta'(I,\chi)$ and $\theta'(I,\chi_1)$, we note that
\[ |\theta'(I,\chi) - \theta'(I,\chi_1)| \leq \sum_{\substack{p \in I\\ p|q}} \log p.\]
Thus, the second quantity in (\ref{eq:BHD2-2}) is
\[ \leq \sum_{q \leq Q} \frac{1}{\phi(q)} \max_{\pi \in \mathcal{P}_x}\left( \sum_{I \in \pi} \sum_{\chi \mod q} \sum_{\substack{p\in I \\ p|q}} \log(p)\right)^2 = \sum_{q \leq Q} \phi(q) \left( \sum_{\substack{p\leq x\\ p|q}} \log(p)\right)^2.\]
Since $\sum_{p|q} \log(p) \leq \log(q)$, this is
\[\leq \sum_{q \leq Q} \phi(q)\log^2(q) \ll Q^2 \log^2(Q) \leq xQ\log^2(x),\]
for $Q\leq x$.

It now suffices to bound the first quantity in (\ref{eq:BHD2-2}). Every primitive character $\chi_1 \mod q_1$ induces characters modulo $q$ for every $q$ that is a multiple of $q_1$. Also, the set of primitive characters $\chi_1$ inducing characters modulo $q$ can be divided into primitive characters modulo each divisor of $q$. By applying the triangle inequality and maximizing separately for each divisor of $q$, we see that for each $q$:
\[ \max_{\pi \in \mathcal{P}_x} \sum_{I  \in \pi} \sum_{\chi \mod q} |\theta'(I, \chi_1)|^2 \leq \sum_{q_1 |q} \max_{\pi \in \mathcal{P}_x} \sum_{I \in \pi} \; \sideset{}{^*}\sum_{\chi_1 \mod q_1} |\theta'(I,\chi_1)|^2.\]
Now summing over $q$ and accounting for the multiple occurrences of each $q_1$, we have
\[
\sum_{q \leq Q} \frac{1}{\phi(q)} \max_{\pi \in \mathcal{P}_x} \sum_{\pi \in I} \; \sum_{\chi \mod q} |\theta'(I,\chi_1)|^2 \leq \sum_{q \leq Q}\max_{\pi \in \mathcal{P}_x} \sum_{I \in \pi} \; \sideset{}{^*}\sum_{\chi \mod q} |\theta'(I,\chi)|^2 \left(\sum_{j\leq \frac{Q}{q}} \frac{1}{\phi(jq)} \right).
\]
As previously noted, this final sum over $j$ is $\ll \phi(q)^{-1} \log(2Q/q)$. Hence the above quantity is
\[ \ll \sum_{q \leq Q} \frac{\log(2Q/q)}{\phi(q)} \max_{\pi \in \mathcal{P}_x} \sum_{I \in \pi} \; \sideset{}{^*}\sum_{\chi \mod q} |\theta'(I,\chi)|^2.\]

As in the proof of Theorem \ref{thm:VarBHD}, we break this sum over $q$ into smaller ranges. For any $1\leq U \leq Q$, we have
\begin{equation}\label{eq:BHD2-3}
\sum_{ U < q \leq 2U} \frac{\log(2Q/q)}{\phi(q)} \max_{\pi \in \mathcal{P}_x} \sum_{I \in \pi} \; \sideset{}{^*}\sum_{\chi \mod q} |\theta'(I,\chi)|^2 \ll U^{-1} \log(2Q/U) \sum_{q\leq 2U} \frac{q}{\phi(q)} \max_{\pi \in \mathcal{P}_x} \sum_{I \in \pi}\; \sideset{}{^*} \sum_{\chi \mod q} |\theta(I,\chi)|^2
\end{equation}
(note that the change of notation from $\theta'$ to $\theta$ does not change any values).
We may now apply Lemma \ref{lem:mid} with $a_p = \log(p)$ for all primes $p$ and $a_n=0$ otherwise.
We conclude that the righthand side of (\ref{eq:BHD2-3}) is
\[ \ll U^{-1} \log(2Q/U)(x+U^2)\log(x) \left(\sum_{p\leq x} \log^2(p)\right) \ll U^{-1}\log(2Q/U) (x+U^2)x \log^2(x).\]

We define $Q_1 = \log^{A+1}(x)$. Setting $U = Q2^{-j}$ and summing over $j$ from 0 to $J := \lceil \log(Q/Q_1)\rceil$, we see that
\[
\sum_{Q_1 < q \leq Q} \frac{\log(2Q/q)}{\phi(q)} \max_{\pi \in \mathcal{P}_x} \sum_{I \in \pi} \; \sideset{}{^*}\sum_{\chi \mod q} |\theta'(I,\chi)|^2 \ll Q^{-1}x \log^2(x) \sum_{j=0}^J 2^j (j+1)(x+Q^2 2^{-2j}).
\]
Since $\sum_{j=0}^\infty (j+1)2^{-j}$ converges and $\sum_{j=0}^J (j+1)2^j \ll J 2^J$, this quantity is $\ll Q_1^{-1} x^2\log^2(x) \log(Q) + xQ\log^2(x)$. Recalling that $Q_1 = \log^{A+1}(x)$ and $x\log^{-A}(x) \leq Q \leq x$, we see this is $\ll xQ\log^2(x)$ as required.

It only remains to bound the contribution from the values of $q \leq Q_1$. We observe that
\[\sum_{q \leq Q_1} \frac{\log(2Q/q)}{\phi(q)} \max_{\pi \in \mathcal{P}_x} \sum_{I \in \pi} \; \sideset{}{^*}\sum_{\chi \mod q} |\theta'(I,\chi)|^2 \ll \log(Q) \sum_{q \leq Q_1} \frac{1}{\phi(q)} \; \sideset{}{^*}\sum_{\chi \mod q} \max_{\pi \in \mathcal{P}_x} \sum_{I \in \pi} |\theta'(I,\chi)|^2.\]
Now, every primitive character modulo $q$ for $q>1$ is non-principal. For the principal character $\chi_0$ modulo 1, the value of $\theta'(I,\chi_0)$ is 0 for any $I$, so we may rewrite our quantity as
\[\log(Q) \sum_{1 <q \leq Q_1} \frac{1}{\phi(q)} \; \sideset{}{^*}\sum_{\chi \mod q} \max_{\pi \in \mathcal{P}_x} \sum_{I \in \pi} |\theta(I,\chi)|^2.\]
Applying Lemma \ref{lem:VarSW} for the constant $A+1$, we see this is
\[ \ll_A \log(Q) \sum_{1 < q \leq Q_1} \frac{1}{\phi(q)} \; \sideset{}{^*}\sum_{\chi \mod q} x^2 e^{-\tilde{c}_A \log^{\frac{1}{2}}(x)},\]
for some positive constant $\tilde{c}_A$ depending only on $A$. Since there are $\phi(q)$ characters modulo $q$ (and hence at most $\phi(q)$ primitive ones), this is $\ll_A Q_1 \log(Q) x^2 e^{-\tilde{c}_A\log^{\frac{1}{2}}(x)}$. Recalling that $Q \leq x$ and $Q_1 = \log^{A+1}(x)$ and noting that $e^{-\tilde{c}_A \log^{\frac{1}{2}(x)}} \ll_A \log^{-2A}(x)$, we see this $\ll_A xQ\log^2(x)$. This completes the proof of Theorem \ref{thm:VarBHD2}.

\section{A Variational Form of the Large Sieve Inequality}\label{sec:varLargeSieve}

We now prove another variational form of Theorem \ref{thm:standardLargeSieve}. This will refine an estimate of Uchiyama stated below. The techniques are similar to those used above. We let $\mathcal{P}_{M,N}$ denote the set of partitions of $[M+1,M+N]$.

\begin{lemma}\label{lem:VarLargeSieve} For all positive integers $Q,M,N$ and complex numbers $\{a_n\}_{n=M+1}^{M+N}$,
\begin{equation}\label{eq:VarLargeSieve}
\sum_{q \leq Q} \frac{q}{\phi(q)} \sideset{}{^*}\sum_{\chi \mod q} \max_{\pi \in \mathcal{P}_{M,N}} \sum_{I \in \pi} |T(\chi,I)|^2 \ll \log^2(N)\left( N + Q^2\right) \sum_{n=M+1}^{N+M} |a_n|^2.
\end{equation}
\end{lemma}

\begin{proof} Without loss of generality, we assume that $N = 2^k$ for some $k$ (rounding $N$ up to the nearest power of 2 can be absorbed by the implicit constant). The interval from $M+1$ to $M+N$ can then be decomposed into dyadic intervals of the form $I_{c,\ell} := (M+(c-1)2^\ell, M+ c2^\ell]$ for each $\ell$ from 0 to $k$, where $c$ ranges from 1 to $2^{k-\ell}$. If we fix $\ell$ and let $c$ vary, we refer to the resulting set of intervals as the dyadic intervals on \emph{level $\ell$}.

For a fixed $q$ and primitive character $\chi$ mod $q$, we consider a maximizing partition $\pi^*$ in $\mathcal{P}_{M,N}$. By Lemma \ref{lem:binarydecomp}, every interval $I \in \pi^*$ can be decomposed as a union of $O(\log N)$ dyadic intervals of the form $I_{c,\ell}$ for varying $c$ and $\ell$ values, such that there are at most 2 intervals included on each level. We let $D(I)$ denote the set of dyadic intervals in the decomposition of $I$. For each $\ell$, $D_\ell(I)$ denotes the subset of intervals in $D(I)$ on level $\ell$.


We can now express the square root of the left hand side of (\ref{eq:VarLargeSieve}) as:
\[\sqrt{ \sum_{q \leq Q} \frac{q}{\phi(q)} \sideset{}{^*}\sum_{\chi \mod q} \sum_{I \in \pi^*}\left| \sum_{J \in D(I)} T(\chi, J)\right|^2}.\]
We can alternatively express this as:
\[
\sqrt{\sum_{q \leq Q} \frac{q}{\phi(q)} \sideset{}{^*}\sum_{\chi \mod q} \sum_{I \in \pi^*}\left|\sum_{\ell=0}^k \; \sum_{J \in D_\ell(I)} T(\chi,J)\right|^2}.
\]
Applying the triangle inequality for the $\ell^2$ norm, this quantity is
\begin{equation}\label{eq:LS1}
\leq \sum_{\ell = 0}^k \sqrt{\sum_{q \leq Q}\frac{q}{\phi(q)} \sideset{}{^*}\sum_{\chi \mod q} \sum_{I \in \pi^*} \left| \sum_{J \in D_\ell(I)} T(\chi,J)\right|^2}.
\end{equation}

For a fixed $\ell$, we consider the quantity
\begin{equation}\label{eq:LS2}
\sum_{q \leq Q}\frac{q}{\phi(q)} \sideset{}{^*}\sum_{\chi \mod q} \sum_{I \in \pi^*} \left| \sum_{J \in D_\ell(I)} T(\chi,J)\right|^2.
\end{equation}
We note that the innermost sum contains at most two intervals $J$, since the dyadic decomposition of each $I$ contains at most two intervals on level $\ell$. Noting that $|a+b|^2 \ll |a|^2 + |b|^2$ holds for all complex numbers $a$ and $b$ and that each dyadic interval on level $\ell$ can occur in the decomposition of at most one $I \in \pi^*$, the quantity in (\ref{eq:LS2}) is
\[ \ll \sum_{q \leq Q} \frac{q}{\phi(q)}\; \sideset{}{^*} \sum_{\chi \mod q}\; \sum_{c=1}^{2^{k-\ell}} \left|T(\chi,I_{c,\ell})\right|^2.\]
Now the innermost sum is simply over the set of dyadic intervals on level $\ell$. Reordering the finite sums, this is
\[ = \sum_{c=1}^{2^{k-\ell}} \sum_{q \leq Q} \frac{q}{\phi(q)}\; \sideset{}{^*} \sum_{\chi \mod q} \left|T(\chi,I_{c,\ell})\right|^2.\]

For each $c$, we can apply Theorem \ref{thm:standardLargeSieve} to obtain:
\[\sum_{q \leq Q} \frac{q}{\phi(q)}\; \sideset{}{^*} \sum_{\chi \mod q} \left|T(\chi,I_{c,\ell})\right|^2 \ll (2^{\ell}+Q^2)\sum_{n \in I_{c,\ell}} |a_n|^2.\]
Thus, the quantity in (\ref{eq:LS2}) is $\ll (2^{\ell} + Q^2) \sum_{n=M+1}^{M+2^k}|a_n|^2.$

Substituting this back into (\ref{eq:LS1}), we see that the square root of the left hand side of (\ref{eq:VarLargeSieve}) is
\[\ll \sum_{\ell=0}^k \sqrt{\left(2^\ell + Q^2\right) \sum_{n=M+1}^{M+2^k}|a_n|^2}.\]
Hence, the left hand side of (\ref{eq:VarLargeSieve}) is
\[\ll \left(\sum_{\ell=0}^k \sqrt{2^\ell + Q^2}\right)^2 \sum_{n=M+1}^{M+2^k} |a_n|^2.\]

Recalling that $2^k = N$ and loosely bounding $2^\ell + Q^2 \leq 2^k + Q^2$ for all $\ell$, we see this is
\[\ll \left(k\sqrt{2^k + Q^2}\right)^2 \sum_{n=M+1}^{M+2^k} |a_n|^2 = k^2 (2^k + Q^2) \sum_{n=M+1}^{M+2^k} |a_n|^2 \ll \log^2(N)(N + Q^2) \sum_{n=M+1}^{M+N} |a_n|^2.\]
\end{proof}

We note that this refines Uchiyama's Maximal large sieve inequality \cite{U}, which states:

\begin{lemma}\label{lem:LargeSieveU} For all positive integers $Q,M,N$ and complex numbers $\{a_n\}_{n=M+1}^{M+N}$,
\begin{equation}
\sum_{q \leq Q} \frac{q}{\phi(q)} \sideset{}{^*}\sum_{\chi \mod q} \max_{I \subseteq [N]} |T(\chi,I)|^2 \ll \log^2(N)\left( N + Q^2\right) \sum_{n=M+1}^{N+M} |a_n|^2.
\end{equation}
\end{lemma}

Montgomery \cite{M} has asked if the $\log^2(N)$ can be removed. We do not have an answer to this question (though we obtain a lower bound on a related quantity in \cite{LL2}). We note that the $\log^2(N)$ cannot be completely removed in our variational refinement.

To see this, we use the lemma below, which follows easily from Lemma 22 in \cite{LL} and the Cauchy-Schwarz inequality:
\begin{lemma}\label{lem:old} Let $c_1,\ldots, c_N$ denote complex numbers $|c_i| \geq \delta$, for some $\delta>0$. Let $X_1,\ldots , X_N$ denote independent Gaussian random variables each with mean $0$ and variance $1$. Then
\[\mathbb{E} \left[  \sup_{\pi \in \mathcal{P}_{N}} \sum_{I \in \pi } \left| \sum_{n \in I } c_n X_n \right|^2  \right] \gg \delta^2 N \log \log (N) \]
\end{lemma}

Strictly speaking, the lemma in \cite{LL} is stated for real $c_1,\ldots,c_N$, but it can easily be deduced for complex numbers by splitting into real and imaginary parts. Now we consider (\ref{lem:VarLargeSieve}) with interval $[N]$ and $a_n=X_n$ for each $n \in [N]$. To apply Lemma \ref{lem:old} for each $q$ and each character modulo $q$, we consider only the indices $n$ such that $n$ is coprime to $q$. On these values, the character will be nonzero. We let $C(q)$ denote the number of these indices for each $q$.
Then, from Lemma \ref{lem:old}, the expectation of the left hand side can be estimated by
\begin{equation}\label{eq:firstestimate}
 \gg \sum_{q \leq Q} q \cdot C(q) \cdot \log \log (C(q)).
\end{equation}

Now, we note that $\sum_{q\leq Q} C(q)$ is equal to the number of pairs $(n,q)$ such that $n$ and $q$ are coprime. This is also equal to $\sum_{n=1}^N Q \frac{\phi(n)}{n}$. By Theorem 330 in \cite{HW}, $\sum_{n=1}^N \phi(n) = \frac{3}{\pi^2} N^2 + O(N \log(N))$, which implies $\sum_{n=1}^N \frac{\phi(n)}{n} \gg N$. Thus, $\sum_{q \leq Q} C(q) \gg QN$. It follows that there exist positive constants $\epsilon, \delta$ such that $C(q) \geq \epsilon N$ for at least $\delta Q$ values of $q$.
Hence, the quantity in (\ref{eq:firstestimate}) is $\gg Q^2N \log\log(N)$, while the expectation of the sum of squares of the $a_n=X_n$ will be $\ll N$. Thus, at least when $N \ll Q^2$, one needs at least a factor of $\log\log(N)$ in (\ref{lem:VarLargeSieve}). Refining the gap between $\log\log(N)$ and $\log^2(N)$ would be interesting.

\section{Acknowledgements}
We thank Jeffrey Vaaler for useful discussions.

\texttt{A. Lewko, Department of Computer Science, The University of Texas at Austin}

\textit{alewko@cs.utexas.edu}
\vspace*{0.5cm}

\texttt{M. Lewko, Department of Mathematics, The University of Texas at Austin}

\textit{mlewko@math.utexas.edu}

\end{document}